\documentclass{amsart}

\usepackage{amsthm}
\usepackage{amsmath}
\usepackage{amssymb}
\usepackage{color}
\usepackage{ulem}
\usepackage[latin1]{inputenc}
\usepackage{mathrsfs}\usepackage{amscd}\usepackage{graphicx}\usepackage{amsfonts}\usepackage{amsxtra}
\usepackage{multirow}
\usepackage{amscd}

\DeclareMathOperator{\trace}{trace}

\DeclareMathOperator{\rank}{rank}

\newtheorem{thm}{Theorem}[section]
\newtheorem{theorem}[thm]{Theorem}

\newtheorem{corollary}[thm]{Corollary}

\newtheorem{lemma}[thm]{Lemma}
\newtheorem{proposition}[thm]{Proposition}

\newtheorem{definition}[thm]{Definition}
\theoremstyle{remark}

\newtheorem{remark}[thm]{Remark}

\newcommand{\cH}{\mathcal H}

\def\<{\langle}
\def\>{\rangle}

\def\be{\begin{equation}}
\def\ee{\end{equation}}
\def\im{\mathrm{im}}

\begin{document}

\title{Tight and random nonorthogonal fusion frames}

\author{Jameson Cahill} 
\address{Mathematics Department, Duke University, Durham, NC 27708}
\email{jcahill@math.duke.edu}
\author{Peter G. Casazza}
\address{Department of Mathematics, University of Missouri, Columbia, MO 65211}
\email{casazzap@missouri.edu}
\author{Martin Ehler}
\address{Department of Mathematics, University of Vienna, Vienna, Austria}
\email{martin.ehler@helmholtz-muenchen.de}
\author{Shidong Li}
\address{Department of Mathematics, San Francisco State University, San Francisco, CA 94132}
\email{shidong@sfsu.edu}

\begin{abstract}
This paper continues the investigation of nonorthogonal fusion frames started in \cite{CCL12}.  First we show that tight nonorthogonal fusion frames a relatively easy to com by.  In order to do this we need to establish a classification of how to to wire a self adjoint operator as a product of (nonorthogonal) projection operators.  We also discuss the link between nonorthogonal fusion frames and positive operator valued measures, we define and study a nonorthogonal fusion frame potential, and we introduce the idea of random nonorthogonal fusion frames.
\end{abstract}


\date{}
\maketitle

\section{Introduction}
Fusion frames were introduced in \cite{MR2066823} and further developed in \cite{MR2419707}.  Recently there has been much activity around the idea of fusion frames, see \cite{FusionChapter} and references therein.  Loosely speaking, a fusion frame is a collection of subspaces $\{W_i\}_{i=1}^m$ all contained in some bigger Hilbert space $\mathcal{H}$ such that any signal $f\in\mathcal{H}$ can be stably reconstructed from the set of orthogonal projections $\{\pi_i f\}_{i=1}^m$, where $\pi_i$ denotes the orthogonal projection from $\mathcal{H}$ onto $W_i$.  Typically we think of the dimension of each subspaces $W_i$ as being much smaller than the dimension of $\mathcal{H}$ so that a high dimensional signal $f$ can be reconstructed from several low dimensional measurements $\{\pi_i f\}_{i=1}^m$.

In \cite{CCL12}, we introduced the idea of nonorthogonal fusion frames in order to achieve sparsity of the fusion frame operator.  The basic observation in \cite{CCL12} is that replacing orthogonal projections $\pi_i$ in the original definition of fusion frames \cite{MR2419707} by non-orthogonal projections ${P}_i$ onto the same subspaces $W_i$ can result in a fusion frame operator which is much sparser.   This is because, for example, one can always choose the null space of the projection ${P}_i$ to contain some basis elements $\{e_{i_j}\}_j$ that are complementary to the subspace $W_i$, and thereby nullify  some columns of ${P}_i$, which in turn results in sparsity of the (new) fusion frame operator.  One further observation which was made in \cite{CCL12} but was not explored very thoroughly there is that tight nonorthogonal fusion frames are much more abundant than tight (orthogonal) fusion frames.  In this continued effort, constructions of tight nonorthogonal fusion frames and nonorthogonal fusion frames of a prescribed fusion frame operator are provided.  

One of the main applications of fusion frames is to sensor networks.  In this setting we have a collection of sensors each of which collect local information and then transmit this information to some central processing station where all of the separate pieces of local information are fused together.  We think of each sensor as collecting data that is contained in some subspace of a common Hilbert space.  To be more specific each sensor is spanned by a sensory frame $\{\varphi_j\}$ given by the elementary transformation (often simple shifts) of the spatial reversal of the sensor's impulse response function \cite{LiYan08}.  The measurement of each sensor is thus given by $\{\< f, \varphi_j\>\}$.  Therefore, the implementation of nonorthogonal projections in sensor networks is feasible in practice by bundling each sensor with an auxiliary sensor which controls the direction of the projection.  For instance, if the sensor has a low-pass characteristic, the auxiliary sensor would need to have the high-pass or band-pass nature, having certain complementary information.  See section 6 of \cite{CCL12} for a more detailed discussion.

We now give a formal definition of nonorthogonal fusion frames.  Throughout this paper, let $\mathcal{H}_n$ denote an $n$-dimensional Hilbert space. 

\begin{definition}
An operator $P:\mathcal{H}_n\rightarrow\mathcal{H}_n$ is called a \textit{projection} if $P^2=P$.  If in addition we have $P^*=P$ then $P$ is called an \textit{orthogonal projection}.
\end{definition}

\begin{definition}
Let $\{P_i\}_{i=1}^m$ be a collection of projections on $\mathcal{H}$ and $\{v_i\}_{i=1}^m$ a collection of positive real numbers.  We say $\{(P_i,v_i)\}_{i=1}^m$ is a nonorthogonal fusion frame for $\mathcal{H}$ if there exist constants $0<A\leq B<\infty$ such that
$$
A\|f\|^2\leq\sum_{i=1}^m v_i^2\|P_if\|^2\leq B\|f\|^2
$$
for every $f\in\mathcal{H}$.  We say it is tight if $A=B$.
\end{definition}

\begin{definition}
Given a nonorthogonal fusion frame we define the nonorthogonal fusion frame operator $S:\mathcal{H}_n\rightarrow\mathcal{H}_n$ by
\begin{equation}\label{operator}
Sf=\sum_{i=1}^m v_i^2P_i^*P_if.
\end{equation}
\end{definition}
We observe that $\{(v_i,P_i)\}$ is tight if and only if $S=\lambda I$ (where $\lambda=A=B$).  Therefore, much of this paper is devoted to studying ways of writing multiples of the identity in the form of the right hand side of equation \eqref{operator}.  We will also usually assume that $v_i=1$ for every $i=1,...,m$

Before leaving the introduction we collect some basic facts about projections and fix some notation that will be used throughout this paper.

\begin{proposition}\label{basic}
Let $P$ be a projection on $\cH$, $W = \im P$, $W^* \equiv (\ker P)^{\perp}= [(I-P)(\cH)]^{\perp}$. Denote by $P^*$ the adjoint of $P$ in $\cH$.   Then:

(1)  $(P^*)^2 = P^*$, and $\im P^*=W^*$, $\ker P^*=W^{\perp}$.

(2)  $W^* = \im P^* = \im P^*P$.

(3)  $P$ is an invertible operator mapping $W^*$ onto $W$.

(4)  $\dim (W)=\dim (W^*)$.
\end{proposition}

\begin{corollary}
Given subspaces $W,W^*$ of $\cH$ with dim $W=$ dim $W^*$, there
is a projection $P$ onto $W$ with $P^*P(\cH)=W^*$ if and only if
$(W^*)^{\perp}\cap W = \{0\}$.
\end{corollary}

Throughout this paper we will always use the notation of Proposition \ref{basic}; \textit{i.e.}, $P$ will always stand for a (nonorthogonal) projection, $W$ will always be the image of $P$, and $W^*$ will always be the image of $P^*$.  Furthermore, we will always use the symbol $\pi_W$ to denote the \textit{orthogonal} projection onto the subspace $W\subseteq\mathcal{H}_n$.

This paper is organized as follows:  In Section \ref{sec_classification}, for a fixed self-adjoint operator $T$ we will classify the projections $P$ for which $T=P^*P$.  In Section \ref{sec_tightFF}, we apply the results of Section \ref{sec_classification} to derive new results on the existence of tight nonorthogonal fusion frames.  In particular, in subsection \ref{2proj} we give a complete classification of tight nonorthogonal fusion frames with 2 projections.  In Section \ref{section:max}, a brief discussion is given to how much positive operator valued measure (POVM)  $\{T_i\}_{i=1}^m$, cf.~\cite{J.M.Renes:2004aa}, deviates from an orthogonal decomposition of the identity operator $I$.  To this end, maximum correlation between projection-induced $T_i$'s are established.  The last section of the article is devoted to the extension of nonorthogonal fusion frames to the context of probabilistic fusion frames.  Random projections and the number of which add a touching flavor to the understanding of the notion of nonorthogonal fusion frames and perhaps the applications thereby induced.

\section{Classification of self adjoint operators via projections}\label{sec_classification}

Let $T:\mathcal{H}_n\rightarrow\mathcal{H}_n$ be a positive, self adjoint, linear operator.  The main point of this section is to classify the set
$$
\Omega(T)=\{P:P^2=P,P^*P=T\}.
$$
The spectral theorem tells us that $T=\sum_{j=1}^n\lambda_j\pi_j$ where the $\lambda_j$'s are the eigenvalues of $T$ and $\pi_j$ is the orthogonal projection onto the one dimensional span of the $j$th eigenvector of $T$.  Therefore $P\in\Omega(T)$ if and only if $P^*P$ has the same eigenvalues and eigenvectors as $T$.  Also note that if $P\in\Omega(T)$ then $\ker(P)=\text{im}(T)^{\perp}$, and since a projection is uniquely determined by its kernel and its image we have a natural bijection between $\Omega(T)$ and the set
$$
\tilde{\Omega}(T):=\{W\subseteq\mathcal{H}_n:\text{im}(P)=W\text{ for some }P\in\Omega(T)\}.
$$
given by
$$
\Omega(T)\ni P\mapsto im(P)\in\tilde{\Omega}(T).
$$
We start with two elementary lemmas.

\begin{lemma}\label{Pe}
Let $P$ be a projection and let $\{e_j\}_{j=1}^k$ be an orthonormal basis of $W^*$ consisting of eigenvectors of $P^*P$ with corresponding nonzero eigenvalues $\{\lambda_j\}$.  Then $\{Pe_j\}_{j=1}^k$ is an orthogonal basis for $W$ and $\|Pe_j\|=\sqrt{\lambda_j}$.
\end{lemma}

\begin{proof}
Just observe that $\langle Pe_j,Pe_{\ell}\rangle=\langle P^*Pe_j,e_{\ell}\rangle=\lambda_j\langle e_j,e_{\ell}\rangle$.
\end{proof}

\begin{lemma}\label{Lemma_eigenP*P}
Let $P$ be a projection and suppose $\lambda$ is an eigenvalue of $P^*P$, $\lambda\neq 0$.  Then $\lambda\geq 1$.  Moreover, $\lambda=1$ if and only if the corresponding eigenvector is in $W \cap W^*$.
\end{lemma}
\begin{proof}
Note that $W^*=\text{im}\ P^*P$, so all eigenvectors of
$P^*P$ corresponding to nonzero eigenvalues are in $W^*$.  Let $x\in W^*$ and write $Px=x+(P-I)x$.
Since $x \perp (I-P)x$,
\begin{equation}\label{norm}
\|Px\|^2=\|x\|^2+\|(P-I)x\|^2\geq\|x\|^2.
\end{equation}
By the same argument on $P^*$ we get $\|P^*Px\|\geq\|Px\|\geq\|x\|$ for all $x\in W^*$.  Therefore, if $P^*Px=\lambda x$ we have that $\lambda\geq 1$.

Finally, by equation (\ref{norm}), $\lambda=1$ if and only if $(I-P)x=0$, or $x=Px \in W$.  Hence $x\in W\cap W^*$.
\end{proof}

The next proposition allows us reduce our problem to the case when $\text{rank}(T)\leq n/2$

\begin{proposition}\label{bigrank}
Let $P$ be a projection, then we can write
$$
P=P'+\pi_{W\cap W^*}
$$
where $\pi_{W\cap W^*}$ is the orthogonal projection onto $W\cap W^*$, and $P'$ is a projection such that all nonzero eigenvalues of $P'^*P'$ are strictly greater than 1.
\end{proposition}
\begin{proof}
First note that Lemma \ref{Lemma_eigenP*P} says that $W\cap W^*=\{x:P^*Px=x\}$.  Now let $W'$ be the orthogonal complement of $W\cap W^*$ in $W$ and let $P'$ be the projection onto $W'$ along $\ker(P)+W\cap W^*$.  Then $P'\pi_{W\cap W^*}=\pi_{W\cap W^*}P'=0$, so $(P'+\pi_{W\cap W^*})^2=P'^2+\pi_{W\cap W^*}^2=P'+\pi_{W\cap W^*}$.  It is clear that $\text{im}(P'+\pi_{W\cap W^*})=W$. Since $\ker{P}=W^{*\perp}\subseteq(W\cap W^*)^{\perp}$ it follows that $\ker(P)\subseteq\ker(P'+\pi_{W\cap W^*})$ so we must have $\ker(P)=\ker(P'+\pi_{W\cap W^*})$.  Therefore $P=P'+\pi_{W\cap W^*}$, and the nonzero eigenvalues of $P'^*P'$ are precisely the nonzero eigenvalues of $P^*P$ which are greater than 1.
\end{proof}

\begin{corollary}
If $\text{rank}(T)=k>\frac{n}{2}$ and $T$ does not have 1 as an eigenvalue with multiplicity at least $k-\lfloor\frac{n}{2}\rfloor$, then $\Omega(T)=\emptyset$.
\end{corollary}

We can now state the main theorem of this section:

\begin{theorem}\label{theorem}
Let $T:\mathcal{H}_n\rightarrow\mathcal{H}_n$ be a positive, self-adjoint operator of rank $k\leq\frac{n}{2}$.  Let $\{\lambda_j\}_{j=1}^k$ be the nonzero eigenvalues of $T$ and suppose $\lambda_j\geq 1$ for $i=1,...,k$ and let $\{e_j\}_{j=1}^k$ be an orthonormal basis of $\text{im}(T)$ consisting of eigenvectors of $T$.  Then
$$
\tilde{\Omega}(T)=\{\text{span}\{\frac{1}{\sqrt{\lambda_j}}e_j+\sqrt{\frac{\lambda_j-1}{\lambda_j}}e_{j+k}\}:\{e_j\}_{j=1}^{2k}\text{ is orthonormal}\}.
$$
\end{theorem}
\begin{proof}
First suppose $W\in\tilde{\Omega}(T)$ and let $P$ be the projection onto $W$ along $\text{im}(T)^{\perp}$.  Be Lemma \ref{Pe} we know that $\{\frac{Pe_j}{\|Pe_j\|}\}_{j=1}^k$ is an orthonormal basis for $W$.  We also know that $\|Pe_j\|=\sqrt{\lambda_j}$ so
$$
\lambda_j=\|e_j\|^2+\|(P-I)e_j\|^2=1+\|(P-I)e_j\|^2
$$
which means
$$
\|(P-I)e_j\|=\sqrt{\lambda_j-1}
$$
so if we set
$$
e_{j+k}=\frac{(P-I)e_j}{\sqrt{1-\lambda_j}},
$$
then $\{e_j\}_{j=1}^{2k}$ is an orthonormal set and
$$
\frac{Pe_j}{\|Pe_j\|}=\frac{1}{\sqrt{\lambda_j}}e_j+\sqrt{\frac{\lambda_j-1}{\lambda_j}}e_{j+k}.
$$

Conversely suppose $W=\text{span}\{\frac{1}{\sqrt{\lambda_j}}e_j+\sqrt{\frac{\lambda_j-1}{\lambda_j}}e_{j+k}\}$ with $\{e_j\}_{j=1}^{2k}$ orthonormal.  Let $P$ be the projection onto $W$ along $\text{im}(T)^{\perp}$.  Notice that $e_j=e_j+\sqrt{\lambda_j-1}e_{j+k}-\sqrt{\lambda_j-1}e_{j+k}$ with $e_j+\sqrt{\lambda_j-1}e_{j+k}\in W$ and $-\sqrt{\lambda_j-1}e_{j+k}\in\text{im}(T)^{\perp}$, so $Pe_j=e_j+\sqrt{\lambda_j-1}e_{j+k}$ for $j=1,...,k$.  Similarly $e_j+\sqrt{\lambda_j-1}e_{j+k}=\lambda_je_j+(1-\lambda_j)e_j+\sqrt{\lambda_j-1}e_{j+k}$ with $\lambda_je_j\in W^*=\text{im}P^*$ and $(1-\lambda_j)e_j+\sqrt{\lambda_j-1}e_{j+k}\in W^{\perp}=\ker(P^*)$, so $P^*Pe_j=\lambda_je_j$ for $j=1,...,k$.  Therefore, $P^*P$ has the same eigenvectors and corresponding eigenvalues as $T$, so $P^*P=T$, and $W\in\tilde{\Omega}(T)$.
\end{proof}

Before proceeding we remark that Theorem \ref{theorem} is independent of our choice of eigenbasis for $T$.  To see this let $\{e_j'\}_{j=1}^k$ be any other eigenbasis for $T$ and let $W=\text{span}\{\frac{1}{\sqrt{\lambda_j}}e'_j+\sqrt{\frac{\lambda_j-1}{\lambda_j}}e_{j+k}\}$ with $\{e'_j\}_{j=1}^{2k}$ orthonormal.  By the second part of the proof of Theorem \ref{theorem} we have that $W\in\tilde{\Omega}(T)$, and so by the first part of the proof we have that in fact $W=\text{span}\{\frac{1}{\sqrt{\lambda_j}}e_j+\sqrt{\frac{\lambda_j-1}{\lambda_j}}e_{j+k}\}$ with $\{e_j\}_{j=1}^{2k}$ orthonormal.

We now state several consequences of Theorem \ref{theorem}.

\begin{corollary}\label{C1}
If $T$ is a positive self-adjoint operator of rank $\le \frac{n}{2}$ with all nonzero
eigenvalues $\ge 1$, then there is a projection $P$ so that $T=P^*P$.
\end{corollary}

\begin{corollary}\label{C2}
If $T$ is a positive self-adjoint operator of rank $\le \frac{n}{2}$, then there is a
projection $P$ and a weight $v>0$ so that $T=v^2P^*P$.
\end{corollary}

\begin{proof}
Let $\lambda_k$ be the smallest non-zero eigenvalue of $T$.  So all nonzero eigenvalues
of $\frac{1}{\lambda_k}T$ are greater than or equal to 1 and by Corollary \ref{C1} there is a projection
$P$ so that $P^*P=\frac{1}{\lambda_k}T$.  Let $v=\sqrt{\lambda_k}$ to finish the proof.
\end{proof}

In the rest of this section we will analyze the case where $\text{rank}(T)>n/2$.

\begin{proposition}
Let $T$ be a positive self-adjoint operator of rank $k>\frac{n}{2}$ with eigenvectors
$\{e_j\}_{j=1}^n$ and respective eigenvalues $\{\lambda_j\}_{j=1}^n$.  The following
are equivalent:

(1)  There is a projection $P$ so that $T=P^*P$.

(2)  The nonzero eigenvalues of $T$ are greater than or equal to 1 and we have
\[
|\{j:\lambda_j >1\}| \le |\{j:\lambda_j=0\}|.
\]
In particular,
\[
|\{j:\lambda_j=1\}| \ge k-\left\lfloor\tfrac{n}{2}\right\rfloor .
\]
\end{proposition}

\begin{proof} Let $A_1 = \{j:\lambda_j >1\}$, $A_2 = \{j:\lambda_j=0\}$, and $A_3 = \{j:\lambda_j=1\}$, and let $\pi_i$ be the orthogonal projection onto $\text{span}\{e_j:j\in A_i\}$ for $i=1,2,3$.

$(1) \Rightarrow (2)$:  By Proposition \ref{bigrank}, we can write
\[ P = P'+\pi_{W\cap W^*},\]
where $\pi_{W\cap W^*}$ is the orthogonal projection onto $W\cap W^*$, and $P'$ is the projection
onto the orthogonal complement $W'$ of $W\cap W^*$ in W along $ker\ P + W\cap W^*$.  Define $W'^*\equiv \im\, P'^*$.  Then $P'$ is an invertible operator from $W'^*$ onto $W'$, $W'^* \perp W\cap W^*$ and  $W'\perp W\cap W^*$, and $W' \cap W'^* = \{0\}$.  Hence,
 \begin{eqnarray*}
  2 \dim\, W'^* &=&  \dim\, W' + \dim\, W'^*\\
  &=& \dim (W'+W'^*) \\
  &\le& \dim\, W'^* + \dim\, \text{span}\{e_j:j\in A_2\}.
 \end{eqnarray*}
 Since $W'^* = span\ \{e_j: j\in A_1\}$, it follows that $|A_1|\le |A_2|$.

$(2)\Rightarrow (1)$:  Let  $T_1=T(\pi_1+\pi_2)$, so $T=T_1+\pi_3$.  By our assumption
\[
\text{rank}\ T_1 \le \frac{n}{2},
\]
and all non-zero eigenvalues of $T_1$ are strictly greater than 1.  By Theorem \ref{theorem}
there is a projection $P'$ so that $P'^*P'=T_1$.  Let $P = P'+\pi_3$.
Then $P'\pi_3 = \pi_3 P' = 0$.  Hence, $P=P^2$ is a projection and
\[ P^*P = P'^*P' + \pi_3 = T_1+\pi_3 = T.\]
\end{proof}

\begin{remark}
Similar to the proof of Corollary \ref{C2}, if $T$ is a positive self-adjoint operator of
rank $> \frac{n}{2}$ with eigenvalues $\{\lambda_1 \ge \cdots \ge \lambda_k >0 =
\lambda_{k+1} = \cdots \lambda_n\}$, then there is a projection $P$ and a weight
$v=\sqrt{\lambda_k}$ so that $T = v^2P^*P$ if and only if
\[ |\{j:\lambda_j >\lambda_k\}| \le |\{j:\lambda_j =0\}|.\]
\end{remark}

\begin{proposition}\label{P1}
Let $T:\mathcal{H}_n\rightarrow\mathcal{H}_n$ be a positive, self adjoint operator of rank $k>\frac{n}{2}$ whose nonzero eigenvalues are all greater than or equal to 1.  If either

(1) $n$ is even, or

(2)  $n$ is odd and $T$ has at least one eigenvalue in the set $\{0,1,2\}$

then there are two projections $P_1$ and $P_2$ such that $T=P_1^*P_1+P_2^*P_2$.
\end{proposition}
\begin{proof}
 Let $\{e_j\}_{j=1}^n$ be an orthonormal basis of $\mathcal{H}_n$ consisting of
eigenvectors of $T$ with respective
eigenvalues $\{\lambda_j\}_{j=1}^n$, in decreasing order.
\vskip12pt
\noindent {\bf Case 1}:  $n$ is even.

Let $V=\text{span}\{e_j\}_{j\in I}$, $|I|=\frac{n}{2}$.  Note that $T=T\pi_V+T\pi_{V^{\perp}}$.  Also, since $T,\pi_V$, and $\pi_{V^{\perp}}$ are all diagonal with respect to $\{e_j\}_{j=1}^n$ it follows that $T$ commutes with both $\pi_V$ and $\pi_{V^{\perp}}$.  Therefore $(T\pi_V)^*=\pi_V^*T^*=\pi_VT=T\pi_V$, so by Theorem \ref{theorem} there is a projection $P_1$ such that $T\pi_V=P_1^*P_1$.  Similarly we can find a projection $P_2$ such that $T\pi_{V^{\perp}}=P_2^*P_2$.
\vskip12pt

\noindent {\bf Case 2}:  $n$ is odd and $T$ has an eigenvalue in the set $\{0,1,2\}$.
\vskip12pt
We will look at the case for each eigenvalue separately.
\vskip12pt
\noindent {\bf Subcase 1}:  $\lambda_n=0$.
\vskip12pt
Let $\mathcal{H}_1=\text{span}\{e_j:1\leq j\leq n-1\}$.  Then $\dim(\mathcal{H}_1)$ is even so we can apply the same argument as above to $\mathcal{H}_1$.

\noindent {\bf Subcase 2}:  $\lambda_n=1$.
\vskip12pt
Define $T_1,T_2$ by
\[ T_1 e_j = \begin{cases} Te_j & \mbox{if}\ \ j=1,2,\ldots,\frac{n-1}{2}\\
 0 & \mbox{otherwise}
 \end{cases} \]
 \[ T_2 e_j = \begin{cases} Te_j & \mbox{if}\ \ j=\frac{n-1}{2}+1,\ldots,n-1\\
 0 & \mbox{otherwise}
 \end{cases}\]
Then $\text{rank}(T_1)=\text{rank}(T_2)=\frac{n-1}{2}<\frac{n}{2}$ so by Corollary \ref{C1}, we can write
 \[ T_i = P_i^*P_i,\ \ i=1,2.\]
 Let $\pi$ be the orthogonal projection of $\cH$ onto $\text{span}\{e_n\}$ and let
 \[ Q = P_2+\pi,\]
which is clearly a projection.  Then we have
 \[ T = P_1^*P_1 + Q^*Q.\]

 \noindent {\bf Subcase 3}:  $\lambda_j=2$ for some $j$.
 \vskip12pt
 Without loss of geberality, re-index $\{\lambda_j\}$ so that $\lambda_n=2$.  Define $T_1,T_2,$ and $\pi$ as above.
 As in the previous case, define two projections $\{P_i\}_{i=1}^2$ so that
 \[ T_i = P_i^*P_i.\]
 Now let $Q_i = P_i+\pi$, $i=1,2$.  Then
 \[ T = Q_1^*Q_1+Q_2^*Q_2.\]
\end{proof}

\begin{corollary}\label{C5}
Let $T:\mathcal{H}_n\rightarrow\mathcal{H}_n$ be a positive, self adjoint operator of rank $k>\frac{n}{2}$.   There is a weight
$v$ and projections $\{P_i\}_{i=1}^2$ so that
\[ T = v^2P_1^*P_1 + v^2P_2^*P_2.\]
\end{corollary}

\begin{proof}
Let $T$ have eigenvectors $\{e_j\}_{j=1}^n$ with respective eigenvalues
$\{\lambda_1 \ge \lambda_2 \ge \lambda_k>0=\lambda_{k+1}=\ldots=\lambda_n\}$.
If $n$ is even, we are done by Proposition \ref{P1}.   If $n$ is odd,
 let $T_1 = \frac{1}{\lambda_k}T$.  Then the smallest eigenvalue of $T_1$ equals 1.
By Proposition \ref{P1}, we can find projections $\{P_i\}_{i=1}^2$ so that
\[ \frac{1}{\lambda_k}T = T_1 = P_1^*P_1+P_2^*P_2.\]
Letting $v=\sqrt{\lambda_k}$ finishes the proof.
\end{proof}

It is important to note that, without weighting,  we can always write every positive self-adoint $T$ as the sum of $P_i^*P_i$ with three projections.

\begin{corollary}
If $T:\mathcal{H}_n\rightarrow\mathcal{H}_n$ is a positive, self adjoint operator of rank $k>\frac{n}{2}$ whose nonzero eigenvalues are all greater than or equal to 1, then
there are projections $\{P_i\}_{i=1}^3$ so that
\[ T = P_1^*P_1+P_2^*P_2+P_3^*P_3.\]
\end{corollary}

\begin{proof}
If $n$ is even, we can write $T$ as the sum of two projections.  Suppose $n$ is odd and let $\{e_j\}_{j=1}^n$ be an eigenbasis of $T$.  Suppose $J_1\cup J_2\cup J_3=\{1,...,n\}$ with $|J_i|<\frac{n}{2}$ and let $\pi_i$ be the orthogonal projection onto $\text{span}\{e_j:j\in J_i\}$ for $i=1,2,3$. Then $T=T(\pi_1+\pi_2+\pi_3)$ and $T\pi_i$ satisfies Corollary \ref{C1} for each $i$.
\end{proof}

We note before leaving this section that the classification results may also be expand a bit more to a set of self-adjoint operators $T$ that are not necessarily positive.

\begin{corollary}
Suppose $T=T_1-T_2$ where $T_1,T_2$ are
positive, self-adjoint operators.  Then there are projections $\{P_i\}_{i=1}^4$
and weights $\{v_i\}_{i=1}^2$ so that
\[ T = v_1^2(P_1^*P_1 + P_2^*P_2) - v_2^2(P_3^*P_3 +P_4^*P_4).\]
\end{corollary}

\section{Tight nonorthogonal fusion frames}\label{sec_tightFF}

In this section we address some issues regarding tight nonorthogonal fusion frames.  The first theorem addresses the issue of which sets of dimensions allow the existence of a tight nonorthogonal fusion frame.  The corresponding problem for fusion frames has received considerable attention proven to be quite difficult, see \cite{MR2671171}, \cite{MR2754774}, and \cite{Bownik}.

\begin{theorem}\label{T1}
Suppose $n_1+\cdots+n_m\geq n$, $n_i\leq\frac{n}{2}$.  Then there exists a tight nonorthogonal fusion frame $\{P_i\}_{i=1}^m$ ($v_i=1$ for every $i$) for $\mathcal{H}_n$ such that $\text{rank}(P_i)=n_i$ for $i=1,...,m$.
\end{theorem}
\begin{proof}
Choose an orthonormal basis $\{e_j\}_{j=1}^n$ for $\mathcal{H}_n$ and choose a collection of subspaces $\{W_i\}_{i=1}^m$ such that: \\
1) $W_i=\text{span}\{e_j\}_{j\in J_i}$ with $|J_i|=n_i$ for each $i=1,...,m$, and \\
2) $W_1+\cdots+W_m=\mathcal{H}_n$. \\
Let $\pi_i$ be the orthogonal projection onto $W_i$ and let $S=\sum_{i=1}^m\pi_i$.  Observe that $I=S^{-1}S=\sum_{i=1}^mS^{-1}\pi_i$.  Since each $\pi_i$ is diagonal with respect to $\{e_j\}_{j=1}^n$ it follows that $S^{-1}$ commutes with $\pi_i$, so $S^{-1}\pi_i$ is positive and self adjoint for every $i=1,...,m$.  Let $\gamma$ be the smallest nonzero eigenvalue of any $S^{-1}\pi_i$, then $\frac{1}{\gamma}S^{-1}\pi_i$ satisfies the hypotheses of Corollary \ref{C1} so there is a projection $P_i$ so that $P_i^*P_i=\frac{1}{\gamma}S^{-1}\pi_i$, and we have
$$
\sum_{i=1}^mP_i^*P_i=\frac{1}{\gamma}I.
$$
\end{proof}

Theorem \ref{T1} should be compared with Theorem 3.2.2 in \cite{MR2671171}.  Also note that the proof of Theorem \ref{T1} is constructive, cf \cite{MR2754774}.  The next theorem deals with adding projections to a given nonorthogonal fusion frame it order to get a tight nonorthogonal fusion frame.  Somewhat surprisingly, this can always be achieved with only two projections.

\begin{theorem}\label{T2}
Let $\{P_i\}_{i=1}^m$ be projections on $\cH_n$, $n\ge 2$.  Then there are two projections
$\{P_i\}_{i=m+1}^{m+2}$ and a $\lambda$ so that
\[ \sum_{i=1}^{m+2}P_i^*P_i = \lambda I.\]
\end{theorem}

\begin{proof}
Let
\[ S = \sum_{i=1}^m P_i^*P_i,\]
and let $\lambda_1 \ge \lambda_2 \ge \cdots \ge \lambda_n\ge 0$ be the
eigenvalues of $S$.  Let $\lambda = \lambda_1+1$ and let
\[ T = \lambda I - S.\]
Then $T$ is a positive self-adjoint operator with all of its eigenvalues
$\ge 1$ and at least one eigenvalue
equal to one.   By Proposition \ref{P1}, we can find projections $\{P_i\}_{i=m+1}^{m+2}$ so
that
\[ T = P_{m+1}^*P_{m+1}+P_{m+2}^*P_{m+2}.\]
Thus,
\[ \lambda I = S+T = \sum_{i=1}^{m+2}P_i^*P_i.\]
\end{proof}

No such theorem exists for frames or regular (orthogonal) fusion frames.  In general we need to add $n-1$ vectors to a frame in $\mathcal{H}_n$ in order to get a tight frame (see Proposition 2.1 in \cite{CT}).  However, in this context Theorem \ref{T2} may be misleading, as the ranks of the projections we need to add could be quite large.  The next result tells us how to deal with the case where we want small rank projections.

\begin{proposition}
If $\{P_i\}_{i=1}^m$ are projections on $\cH_n$ and $k\le \frac{n}{2}$, there
are projections $\{Q_i\}_{i=1}^{L}$ with $L= \lceil \frac{n}{k} \rceil$
and $\text{rank}(Q_i)\leq k$, and a $\lambda$ so that
\[ \sum_{i=1}^m P_i^*P_i + \sum_{j=1}^LQ_i^*Q_i = \lambda I.\]
\end{proposition}

\begin{proof}
Let $S$ $S=\sum_{i=1}^m P_i^*P_i$ and assume $S$ has eigenvectors $\{e_j\}_{j=1}^n$ with respective
eigenvalues $\lambda_1 \ge \lambda_2 \ge \cdots \ge \lambda_n$.  Partition the set $\{1,...,n\}$ into sets $J_1,...,J_L$ with $|J_{\ell}|\leq k$ for every $\ell=1,...,L$.  Let $\pi_{\ell}$ denote the orthogonal projection onto $\text{span}\{e_j\}_{j\in J_{\ell}}$.  Set $\lambda=\lambda_1+1$ and let $T_{\ell}=(\lambda I-S)\pi_{\ell}$.  Then each $T_{\ell}$ satisfies the hypotheses of Corollary \ref{C1} so choose any projection $Q_{\ell}\in\Omega(T_{\ell})$.  Now we have that
\begin{eqnarray*}
\sum_{i=1}^M P_i^*P_i+\sum_{\ell=1}^LQ_{\ell}^*Q_{\ell}&=&S+\sum_{\ell=1}^LT_{\ell} \\
&=& S+\lambda I-S=\lambda I.
\end{eqnarray*}
\end{proof}

\subsection{2 projections}\label{2proj}

As an application of the results of the previous section we will give a complete description of when there are two projections $P_i:\mathcal{H}_n\rightarrow\mathcal{H}_n$,
$i=1,2$ such that
\begin{equation}\label{2}
P_1^*P_1+P_2^*P_2=\lambda I.
\end{equation}
Let $W_1=\text{im}(P_1),W_1^*=\text{im}(P_1^*),W_2=\text{im}(P_2),W_2^*=\text{im}(P_2^*)$.  We will examine this in several cases but first we make some general remarks.  Note that if $x\in W_1^*$ such that $P_1^*P_1x=\alpha x$ (for $\alpha\in\mathbb{R}$) then $P_2^*P_2x=(\lambda-\alpha)x$, so there is an orthonormal bases $\{e_j\}_{j=1}^n$ consisting of eigenvectors of both $P_1^*P_1$ and $P_2^*P_2$.  Furthermore, if $P_1^*P_1x=0$ then $P_2^*P_2x=\lambda x$, so $\ker{P_1}=W_1^{*\perp}\subseteq W_2^*$, and similarly $W_2^{*\perp}\subseteq W_1^*$.

It follows from (\ref{2}) that $\text{rank}(P_1)+\text{rank}(P_2)\geq n$.  We will examine the cases of equality and strict inequality separately.

\begin{proposition}
Suppose $P_1$ and $P_2$ are projections on $\mathcal{H}_n$ such that $P_1^*P_1+P_2^*P_2=\lambda I$ and that $\text{rank}(P_1)+\text{rank}(P_2)=n$.  Then either $\text{rank}(P_1)\neq\text{rank}(P_2)$ and $\lambda=1$ or $\text{rank}(P_1)=\text{rank}(P_2)=\frac{n}{2}$ and $\lambda\geq 1$.
\end{proposition}
\begin{proof}
First suppose without loss of generality that $\text{rank}(P_1)=k>\text{rank}(P_2)$.  In this case we have that $k>\frac{n}{2}$, so $\dim(W_1\cap W_1^*)\geq 2k-n>0$.  Then by Proposition \ref{bigrank} we have that $P_1=P_1'+\pi_{W_1\cap W_1^*}$ and $P_1^*P_1+P_2^*P_2=P_1'^*P_1'+\pi_{W_1\cap W_1^*}+P_2^*P_2$.  Let $x\in W_1\cap W_1^*$, then $P'x=0$, and since $x\not\in W_2^*$ it follows that $P_2x=0$.  Therefore $(P_1^*P_1+P_2^*P_2)x=x$ which means $\lambda=1$, both $P_1$ and $P_2$ are orthogonal projections, and $W_j^*=W_j$ $j=1,2$.

Now suppose that $n$ is even, and $\dim(W_1)=\dim(W_2)=\frac{n}{2}$.  In this case we have that $W_1^*=W_2^{*\perp}$, so it follows immediately that $P_1^*P_1=\lambda\pi_{W_1^*}$ and $P_2^*P_2=\lambda\pi_{W_2^*}$.
\end{proof}

\begin{proposition}
Suppose $P_1$ and $P_2$ are projections on $\mathcal{H}_n$ such that $P_1^*P_1+P_2^*P_2=\lambda I$ and that $\text{rank}(P_1)+\text{rank}(P_2)>n$.  Then $\text{rank}(P_1)=\text{rank}(P_2)$, $\lambda=2$, and $W_1^*\cap W_1=W_2^*\cap W_2$.
\end{proposition}
\begin{proof}
First suppose $\dim(W_1)=k>\ell=\dim(W_2)$.  Note that $k>\frac{n}{2}$.  By the remarks above we know that $0$ must be an eigenvalue of $P_1^*P_1$ with multiplicity $n-k$, $\lambda$ must be an eigenvalue of $P_1^*P_1$ with multiplicity $n-\ell$, and $1$ must be an eigenvalue of $P_1^*P_1$ with multiplicity $\dim(W_1^*\cap W_2^*)\geq 2k-n$.  Adding up these multiplicities we get $(n-k)+(n-\ell)+(2k-n)=n+k-\ell=n$ which contradicts the fact that $k>\ell$.  Therefore, we may assume that $\dim(W_1)=\dim(W_2)$.

By the remarks above we can choose an orthonormal basis $\{e_j\}_{j=1}^n$ of $\mathcal{H}_n$ so that
\begin{eqnarray*}
P_1^*P_1e_j&=&\lambda e_j\text{ and }P_2^*P_2e_j=0\text{ for }j=1,...,n-k, \\
P_1^*P_1e_j&=&0\text{ and }P_2^*P_2e_j=\lambda e_i\text{ for }j=k+1,...,n.
\end{eqnarray*}
Since $\dim(W_1\cap W_1^*),\dim*(W_2\cap W_2^*)\geq 2k-n$ it follows that
$$
P_1^*P_1e_j=e_j=P_2^*P_2e_j\text{ for }j=n-k+1,...,k.
$$
Therefore $\lambda=2$ and $W_1\cap W_1^*=W_2\cap W_2^*$.
\end{proof}

Ideally we would like analogous theorems for any number of projections, but this seems to be quite a difficult problem.

\section{Maximal correlation between projections}\label{section:max}

Any tight nonorthogonal fusion frame $\{(P_i,v_i)\}_{i=1}^m$ induces a collection of self-adjoint, positive semi-definite operators $\{T_i\}_{i=1}^m=\{v^2_i P^*_iP_i\}_{i=1}^m$ satisfying $\sum_{i=1}^m T_i = AI$, where $A>0$ is the frame bound. For simplicity, we shall assume that the weights $\{v_i\}_{i=1}^m$ are scaled such that $A=1$. In this case, $\{T_i\}_{i=1}^m$ is also called a positive operator valued measure (POVM), cf.~\cite{J.M.Renes:2004aa}.

To study how much $\{T_i\}_{i=1}^m$ deviates from an orthogonal decomposition of the identity operator $I$, we aim to estimate
\begin{equation}\label{eq:correlation}
\max_{i\neq j} \langle T_i,T_j\rangle,
\end{equation}
from below. Here, $\langle T,R\rangle:=\trace(T^*R)$ is the standard inner product between any two linear operators $T$ and $R$ on $\mathcal{H}_n$, so that the Hilbert-Schmidt norm $\|T\|_{HS}$ of $T$ is induced by $\|T\|^2_{HS}:=\langle T,T\rangle$.

We obtain lower estimates on \eqref{eq:correlation} by generalizing the simplex bound as derived in \cite{Conway:1996aa}  for orthogonal projections of equal rank and later extended in \cite{Bachoc:2010aa} to orthogonal projections of mixed ranks:
\begin{proposition}\label{prop:simplex}
If $\{T_i\}_{i=1}^m$ is a collection of self-adjoint, positive semi-definite operators on $\mathcal{H}_n$ scaled such that $\sum_{i=1}^m \trace(T_i)=n$, then
\begin{equation}\label{eq:simplex 2}
\max_{i\neq j} \langle T_i,T_j\rangle \geq \frac{n-\sum_{i=1}^m\langle T_j,T_j\rangle}{m(m-1)}.
\end{equation}
Equality holds if and only if $\{T_i\}_{i=1}^m$ is equiangular and satisfies $\sum_{i=1}^m T_i=I$.
\end{proposition}
Here, we say that $\{T_i\}_{i=1}^m$ is equiangular if $\langle T_i,T_j\rangle$ does not depend on the choice of $i\neq j$. To verify Proposition \ref{prop:simplex}, we can follow the lines in \cite{Bachoc:2010aa}, so that the detailed proof is omitted here.

Being equiangular is a strong requirement, and we find a bound on the maximal number of operators that satisfy Theorem \ref{prop:simplex} with equality. The analogue of the following theorem by means of orthogonal projectors was derived in \cite{Bachoc:2010aa}:
\begin{theorem}\label{th:equiangular bounds}
If $\{T_i\}_{i=1}^m$ is a collection of pairwise distinct, equiangular, self-adjoint, positive semi-definite operators on $\mathcal{H}_n$ such that $\sum_{i=1}^m T_i=I$ and  $\| T_i\|^2_{HS}=:c$ does not depend on $i$, then $m\leq n^2$. If, in addition, $\mathcal{H}_n$ is a real Hilbert space, then $m\leq 1/2 n(n+1)$.
\end{theorem}
\begin{proof}
Let $\beta:=\langle T_i,T_j  \rangle\neq c$, $i\neq j$. Cauchy-Schwartz implies that $\beta<c$. Since $\sum_{i,j}T_iT_j = I$, applying to both sides the trace yields $mc+m(m-1)\beta= n$ or equivalently
\begin{equation}\label{eq:1}
\beta = \frac{n-mc}{m(m-1)}.
\end{equation}
 In order to verify that $\{T_i\}_{i=1}^m$ is linearly independent in the collection of linear operators on $\mathcal{H}_n$, we assume that $0=\sum_{i=1}^m\alpha_i T_i$. Applying $T_j$ and taking the trace yields $0=\alpha_i(c-\beta)+\beta\sum_{i=1}^m\alpha_j$. Since $c\neq \beta$, we derive that $\alpha_j=\frac{\beta}{\beta-c}\sum_{i=1}^m\alpha_i=:\gamma$ does not depend on the choice of $j$. We assume that $\gamma\neq 0$ and obtain $(m-1)\beta= -c$. By using \eqref{eq:1}, the latter induces $n=0$. Thus, $\gamma$ must be zero, which implies that $\alpha_i=0$, for all $i=1,\ldots,m$. Finally, we have shown that $\{T_i\}_{i=1}^m$ is linearly independent in the space of linear operators on $\mathcal{H}_n$. Since $\mathcal{H}_n$ is $n$-dimensional, we obtain $m\leq n^2$.

 If $\mathcal{H}_n$ is a real Hilbert space, then the collection of self-adjoint operators form an $1/2n(n+1)$ dimensional subspace of the linear operators on $\mathcal{H}_n$, so that we derive $m\leq  1/2n(n+1)$.

\end{proof}

\section{Random nonorthogonal fusion frames}
Probabilistic versions of frames have been introduced in \cite{Ehler:2010aa,Okoudjou:2010aa,Ehler:2011uq}. Here, we extend the concept to nonorthogonal fusion frames.

Let $\Omega$ be a locally compact Hausdorff space and $\mathcal{B}(\Omega)$ be the Borel-sigma algebra on $\Omega$ endowed with a probability measure $\mu$.
We denote the collection of projections on $\mathcal{H}_n$ by $\mathcal{P}_{n}$, endowed with the induced Borel sigma algebra. We say that a random projector $P:\Omega \rightarrow \mathcal{P}_{n}$, is a random nonorthogonal fusion frame if there are nonnegative constants $A$ and $B$ such that
\begin{equation*}
A\|x\|^2 \leq \int_{\Omega} \|P(w)x\|^2 d\mu(\omega) \leq B\|x\|^2,\quad\text{for all $x\in\mathcal{H}_n$.}
\end{equation*}
The random projector $P$ is called tight if we can choose $A=B$. The random analysis operator $F$ is defined by
\begin{equation*}
F : \mathcal{H}_n \rightarrow L_2(\Omega,\mathcal{H}_n,\mu),\quad x\mapsto (\omega\mapsto P(\omega)x).
\end{equation*}
Its adjoint operator $T^*$ is the random synthesis operator
\begin{equation*}
F^* : L_2(\Omega,\mathcal{H}_n,\mu) \rightarrow \mathcal{H}_n,\quad f\mapsto \int_{\Omega} P(\omega)^*f(\omega) d\mu(\omega).
\end{equation*}
The random nonorthogonal fusion frame operator $S=F^* F$ then is
\begin{equation*}
S :\mathcal{H}_n\rightarrow\mathcal{H}_n,\quad x\mapsto \int_{\Omega} P(\omega)^* P(\omega) x d\mu(\omega).
\end{equation*}
Thus, $S=\int_{\Omega} P(\omega)^* P\omega  d\mu(\omega)$ has a matrix representation
$\big(\sum_{i=1}^m \langle P^*_iP_ie_k,e_l\rangle\big)_{i,j}$, where $\{e_j\}_{j=1}^n$ is an orthonormal basis for $\mathcal{H}_n$. Moreover, we obtain
\begin{equation*}
A n =   \sum_{j=1}^n A \|e_j\|^2 = \sum_{j=1}^n \int_{\Omega} \|P(\omega)e_j\|^2 d\mu(\omega)\\
 =  \int_{\Omega}\sum_{j=1}^n \langle P(\omega)e_j,P(\omega)e_j\rangle d\mu(\omega).
\end{equation*}
Thus, if $\mu$ is a tight random nonorthogonal fusion frame for $\mathcal{H}_n$, then the frame bound $A$ equals $\frac{1}{n}\int_{\Omega} \trace(P(\omega)^*P(\omega)) d\mu(\omega)$.

Next, we present a construction of tight (random) nonorthogonal fusion frames that is based on finite groups. Recall that a finite subgroup $G$ of the unitary operators $O(\mathcal{H}_n)$ is called \textit{irreducible} if each orbit $Gx$, for $0\neq x\in\mathcal{H}_n$, spans $\mathcal{H}_n$. In other words, any $G$-invariant subspace is trivial, i.e., either $\mathcal{H}_n$ or $\{0\}$. The following result is a generalization of Theorem 6.3 in \cite{Vale:2005aa}, where finite frames were considered:
\begin{theorem}\label{theorem:construct random tight}
If $P$ is a nontrivial random projection and $G$ is an irreducible finite subgroup of $O(\mathcal{H}_n)$, then $\frac{1}{|G|}\sum_{g\in G} g^*Pg$ is a tight random nonorthogonal fusion frame.
\end{theorem}
\begin{proof}
One can directly check that the fusion frame operator $S:\mathcal{H}_n\rightarrow\mathcal{H}_n$,
\begin{equation*}
x\mapsto \frac{1}{|G|}\sum_{g\in G} \int_{\Omega}g^*P(\omega)^*gg^*P(\omega)gxd\mu(\omega)=\frac{1}{|G|}\sum_{g\in G} \int_{\Omega}g^*P(\omega)^*P(\omega)g xd\mu(\omega)
\end{equation*}
 is self-adjoint and positive semi-definite. Since the identity is an element in $G$ and $P$ is not the trivial random projection, $S$ cannot be the zero mapping, so that it has a positive eigenvalue $\lambda$. One checks that each $g\in G$ commutes with $S$. Thus, the $\lambda$-eigenspace is a $G$-invariant subspace. The irreducibility implies that the eigenspace is the full space $\mathcal{H}_n$, so that $S$ is a multiple of the identity.
\end{proof}

For the sake of completeness, we also formulate Theorem \ref{theorem:construct random tight} in terms of finite nonorthogonal fusion frames:
\begin{corollary}
If $P$ is a projection and $G$ an irreducible finite subgroup of $O(\mathcal{H}_n)$, then $\{g^*Pg\}_{g\in G}$ is a tight nonorthogonal fusion frame.
\end{corollary}

Next, we revisit some ideas of Section \ref{section:max} and discuss correlations in a random regime.
If $P$ is a random projection, then we call
\begin{equation*}
\mathcal{R}(P)=\int_{\Omega}\int_{\Omega} \langle P(\omega)^* P(\omega) ,  P(\omega')^* P(\omega')\rangle  d\mu(\omega) d\mu(\omega')
\end{equation*}
its \textit{random nonorthogonal fusion frame potential}.
\begin{proposition}\label{theorem:pot prob general}
If $P$ is a nontrivial random projection, then
\begin{equation}\label{eq:inequality PNFFP}
\mathcal{R}(P)\geq \frac{M^2}{n},\quad \text{where } M=\int_{\Omega} \|P(\omega)\|^2_{HS}d\mu(\omega),
\end{equation}
and equality holds if and only if $P$ is tight.
\end{proposition}
Note that $\mathcal{R}(P) = \trace(S^2)$, where $S$ is the random nonorthogonal fusion frame operator of $P$. This way we see that Proposition \ref{theorem:pot prob general} can be proven by following the lines of the analogues results for orthogonal projectors in \cite{Bachoc:2010aa}.

If $P$ is a projection, then $ \|P\|^2_{HS}\geq\rank(P)$, and equality holds if and only if $P$ is an orthogonal projection.  We therefore have the following:
\begin{corollary}
If $P$ is a nontrivial random projection, then
\begin{equation}\label{eq:fpt}
\mathcal{R}(P)\geq \frac{M^2}{n},\quad \text{where } M=\int_{\Omega} \rank(P(\omega)) d\mu(\omega),
\end{equation}
and equality holds if and only if $P$ is tight and an orthogonal projection almost everywhere.
\end{corollary}

We can expect that the sample of a tight random nonorthogonal fusion frame approximates a tight nonorthogonal fusion frame when the sample size increases. The following theorem generalizes results in \cite{Ehler:2010aa}:
\begin{proposition}\label{th:variances}
Let $\{P_i\}_{i=1}^m$ be a collection of independent tight random nonorthogonal fusion frames with frame bounds $\{A_i\}_{i=1}^m$, respectively, such that,
\begin{equation*}
M:=\frac{1}{m}\sum_{i=1}^m \int_{\Omega} \|P^*_i(\omega)P_i(\omega)\|^2_{HS} d\mu(\omega)<\infty.
\end{equation*}
If $S(\omega)=\sum_{i=1}^m P_i(\omega)^* P_i(\omega)$ denotes the nonorthogonal fusion frame operator associated to $\{P_i(\omega)\}_{i=1}^m$, then
\begin{equation*}
\mathbb{E}(\| \frac{1}{m}S - A I\|^2_{HS}) = \frac{1}{m}(M-n\tilde{A}),
\end{equation*}
where $A=\frac{1}{m}\sum_{i=1}^m A_i$ and $\tilde{A}= \frac{1}{m}\sum_{i=1}^m A^2_i$.
\end{proposition}
\begin{proof}
The $(k,l)$-th entry of the matrix $S$ is given by
$
S_{k,l} = \sum_{i=1}^m \langle P_i^*P_ie_k,e_l\rangle,
$
and we observe that $\mathbb{E}(\langle P_i^*P_ie_k,e_l\rangle))=A_i\delta_{k,l}$. We derive
\begin{align*}
\mathbb{E}(\| \frac{1}{m}S - A I_d\|^2_{HS}) & = \mathbb{E}(\sum_{k,l} \big(\frac{1}{m}S_{k,l} - A \delta_{k,l}\big)^2 )\\
& = \mathbb{E}(\sum_{k,l} \frac{1}{m^2} (S_{k,l})^2)   -\mathbb{E}(\sum_{k}\frac{2A}{m}S_{k,k}) + nA^2 \\
& =\mathbb{E}(\sum_{k,l} \frac{1}{m^2} (S_{k,l})^2)   -nA^2
\end{align*}
since $\mathbb{E}(\sum_k \frac{1}{m} S_{k,k})=nA$. We split the occurring double sum of $(S_{k,l})^2$ into its diagonal and nondiagonal parts so that the independence of $\{P_i\}_{i=1}^m$ yields
\begin{align*}
\mathbb{E}(\| \frac{1}{m}S - A I\|^2_{HS})& = \frac{1}{m}M +\sum_{k,l} \frac{1}{m^2}\sum_{i\neq j} \mathbb{E}(\langle P_j^*P_je_k,e_l\rangle) E(\langle P_i^*P_ie_k,e_l\rangle)   -nA^2\\
& = \frac{1}{m}M +\sum_{k,l} \frac{1}{m^2}\sum_{i\neq j} A_j\delta_{k,l} A_i\delta_{k,l}   -nA^2\\
& =  \frac{1}{m}M+  \frac{n}{m^2}\sum_{i\neq j} A_j A_i   -nA^2=  \frac{1}{m}M-\frac{n}{m}\tilde{A}.\qedhere
\end{align*}
\end{proof}

In order to replace the expectation used in Proposition \ref{th:variances} with a proper estimate of the norm of the difference, we can apply large deviation bounds. For instance, the Matrix Rosenthal inequality as stated in \cite{Mackey:2012uq} implies the following family of estimates:
\begin{theorem}
Let $\{P_i\}_{i=1}^m$ be a collection of independent random tight nonorthogonal fusion frames with frame bound $A$, such that, $\mathbb{E}\|P^*_i P_i - AI\|_{4p}^{4p}<\infty$. Let $S$ be as in Proposition \ref{th:variances}. If $p=1$ or $p\geq 3/2$, then
\begin{equation*}
\mathbb{E}\|\frac{1}{m}S -  A I\|_{4p}^{4p}\leq (\frac{4p-1}{m^2})^{2p}\| \big( \sum_{i=1}^m \mathbb{E}(P^*_i P_i - AI)^2 \big)^{1/2} \|^{4p}_{4p} +(\frac{4p-1}{m})^{4p} \sum_{i=1}^m\mathbb{E}\|P^*_i P_i -AI\|_{4p}^{4p}.
\end{equation*}
\end{theorem}

\end{document}